\documentclass[11pt,a4paper]{amsart}

\usepackage{geometry}
\usepackage{amsmath,amssymb,amsthm}
\usepackage{mathtools}
\usepackage{enumitem}
\usepackage{hyperref}
\usepackage{microtype}

\usepackage[utf8]{inputenc}
\usepackage[english]{babel}
\usepackage{amsmath,amssymb,amsthm}
\usepackage{geometry}
\usepackage{hyperref}

\hypersetup{
    colorlinks=true,
    linkcolor=blue,
    citecolor=blue,
    urlcolor=blue
}

\newtheorem{theorem}{Theorem}[section]

\newtheorem{proposition}[theorem]{Proposition}
\newtheorem{corollary}[theorem]{Corollary}

\theoremstyle{definition}
\newtheorem{definition}[theorem]{Definition}
\newtheorem{remark}[theorem]{Remark}

\newcommand{\R}{\mathbb{R}}
\newcommand{\Vol}{\operatorname{Vol}}
\newcommand{\Vrel}{V_{\operatorname{rel}}}
\newcommand{\E}{\mathbb{E}}

\newcommand{\erre}{\mathbb{R}}
\newcommand{\ene}{N}

\title{\textbf{Measure Concentration Around Minimal Hypersurfaces}}

\author{Vicent Gimeno i Garcia}
\address{Department of Mathematics, Universitat Jaume I-IMAC, E-12071, Castell\'{o}, Spain}
\email{gimenov@mat.uji.es}

\author{Vicente Palmer}
\address{Department of Mathematics, Universitat Jaume I-INIT, E-12071, Castell\'{o}, Spain}
\email{palmer@mat.uji.es}

\thanks{Research partially supported by the Research Program of University Jaume I Project GACUJIMB/2025/11, and by the project AICO/2023/035 funded by Conselleria de Educaci\'o, Cultura, Universitats i Ocupaci\'o-Generalitat Valenciana}



\date{\today}

\begin{document}

\maketitle

\begin{abstract}
The classical measure concentration phenomenon on the standard sphere $\mathbb{S}^n$ asserts that the volume of any tubular neighborhood of a totally geodesic equator $\mathbb{S}^{n-1}$ concentrates almost the entire volume of $\mathbb{S}^n$ as the dimension $n$ grows to infinity. In this paper, we extend this \lq\lq Fat Equator Effect'' to arbitrary closed minimal hypersurfaces $\Sigma^{n-1} \subset \mathbb{S}^n$. By combining the analytic framework of Brownian motion and its Mean Exit Time (MET) function with the geometric tool of Schwarz symmetrization, we establish that minimality alone is sufficient to guarantee measure concentration around $\Sigma^{n-1}$ in high dimensions.
\end{abstract}
\bigskip
\bigskip



\section{Introduction}

The classical Measure Concentration Theorem on the unit sphere $\mathbb{S}^n \subset \R^{n+1}$ is a foundational result in geometric analysis and asymptotic convex geometry.  Formulated by V. Milman and G. Schechtman in \cite{MilmanSchechtman2001} as a direct consequence of L\'evy\'s Isoperimetric Inequality, this result asserts that if a domain $A \subset \mathbb{S}^n$ occupies at least half of the total volume of the sphere, its metric enlargements will cover almost the entire volume of the sphere when the dimension grows to infinity.

Specifically, given a set $A \subset \mathbb{S}^n$, its $\epsilon$-fattening (or $\epsilon$-enlargement) is defined as
\begin{equation}
    A_\epsilon := \left\{ x \in \mathbb{S}^n : \text{dist}^{\mathbb{S}^n}(x, A) < \epsilon \right\}.
\end{equation}
The classical Measure Concentration Theorem is stated as follows:

\begin{theorem}[Measure Concentration in the Sphere \cite{MilmanSchechtman2001}]\label{measurecon}
Let $A \subset \mathbb{S}^n$ be a measurable domain satisfying $\Vol(A) \ge \frac{1}{2} \Vol(\mathbb{S}^n)$. Then, for any $\epsilon \in (0, \pi/2]$,
\begin{equation}
    \frac{\Vol(A_\epsilon)}{\Vol(\mathbb{S}^n)} \ge 1 - \sqrt{\frac{\pi}{8}} e^{-\epsilon^2 (n-1)/2}.
\end{equation}
In particular,
\begin{equation}
    \lim_{n \to \infty} \frac{\Vol(A_\epsilon)}{\Vol(\mathbb{S}^n)} = 1.
\end{equation}
\end{theorem}

A consequence of this concentration phenomenon is the so-called \emph{Fat Equator Effect}. Let $\mathbb{S}^{n-1}$ be a totally geodesic equator in $\mathbb{S}^n$, and denote by $T_\delta(\mathbb{S}^{n-1})$ its tubular neighborhood of radius $\delta \in (0, \pi/2)$, given by
\begin{equation}
    T_\delta(\mathbb{S}^{n-1}) := \left\{ x \in \mathbb{S}^n : \text{dist}_{\mathbb{S}^n}(x, \mathbb{S}^{n-1}) \le \delta \right\}.
\end{equation}
\noindent Then, we have that in high dimensions, virtually all the measure of the sphere $\mathbb{S}^n$ concentrates in an arbitrarily thin strip around any equator $\mathbb{S}^{n-1}$.

\begin{theorem}[Fat Equator Effect, see \cite{Bengtsson}]
For every equatorial tube $T_\delta(\mathbb{S}^{n-1}) \subset \mathbb{S}^n$ of radius $\delta \in (0, \pi/2)$,
\begin{equation}\label{ineqfat}
    \frac{\Vol(T_\delta(\mathbb{S}^{n-1}))}{\Vol(\mathbb{S}^n)} \ge 1 - 2 e^{-\frac{n-2}{2} \delta^2},
\end{equation}
which implies that
\begin{equation}\label{ineqfat2}
    \lim_{n \to \infty} \frac{\Vol(T_\delta(\mathbb{S}^{n-1}))}{\Vol(\mathbb{S}^n)} = 1.
\end{equation}
\end{theorem}

\begin{remark}\label{rem1}
We observe at this point that the existence, for every dimension $n \in \ene$,  of the tube  $T_\delta(\mathbb{S}^{n-1})$, (which provides a precise meaning to the limit \eqref{ineqfat2}), is given by the following considerations. For each $n \in \ene$ we are considering the totally geodesic  immersion $\tilde\varphi_n: (\mathbb{S}^{n-1}, \tilde\varphi_{n}^*(g_{\mathbb{S}^{n}})) \rightarrow (\mathbb{S}^{n},g_{\mathbb{S}^{n}}) $. Then, its minimal focal distance in $\mathbb{S}^{n}$, $\tilde{\mathbb{C}}_n=\mathbb{C}_{\mathbb{S}^{n}}(\mathbb{S}^{n-1})$, (see definition \ref{minfoc} in the following section \S.2) do not depends on the dimension $n$ and we have, for all $n \in \ene$, that
$$\tilde{\mathbb{C}}_n=\mathbb{C}_{\mathbb{S}^{n}}(\mathbb{S}^{n-1})=\frac{\pi}{2}$$ 

Hence, the radius $\delta$ in equation \eqref{ineqfat2} is bounded above by 
$$\delta \leq \inf_{n \in \ene}\{\tilde{\mathbb{C}}_n\}=\frac{\pi}{2}$$

\noindent We shall return on these considerations in the statement of our main result.

\end{remark}

Fat equator effect comes from the concentration of measure stated in Theorem \ref{measurecon} in the following way: let us denote as $\mathbb{S}^{n+}$ the half sphere centered at the north pole and  as $\mathbb{S}^{n-}$ the half sphere centered at the south pole. Then, the $\epsilon$-strip around the equator can be expressed, in terms of the $\epsilon$-fattenings $\mathbb{S}^{n+}_{\epsilon}$ and $\mathbb{S}^{n-}_{\epsilon}$ as the intersection, 
$$\Omega_\epsilon=\mathbb{S}^{n+}_{\epsilon} \cap\mathbb{S}^{n-}_{\epsilon} \,.$$

\noindent Now, inequality (\ref{ineqfat}) follows applying Theorem \ref{measurecon} to the $\epsilon$-fattenings $\mathbb{S}^{n+}_{\epsilon}$ and $\mathbb{S}^{n-}_{\epsilon}$ and using the fact that ${\rm vol}(\mathbb{S}^{n+}_\epsilon)={\rm vol}(\mathbb{S}^{n-}_\epsilon)$. The conclusion is that for  \lq\lq very large'' dimension,  almost all measure in the Sphere is concentrated surrounding  the equator. 

In the previous paper \cite{GP}, we have explored the fat equator effect from the point of view of the submanifold theory. Namely, we have proved that the {\em intersection} of any equatorial band in a $n$-dimensional sphere $\mathbb{S}^n$ with a closed minimal submanifold $\Sigma^m \subseteq \mathbb{S}^n$ acumulates area in the submanifold $\Sigma$ when the dimension $n$ goes to infinity. 

It is natural, therefore, to consider the intrinsic counterpart of this issue:
\medskip

\emph{Does the same concentration phenomenon hold if we replace the totally geodesic equator $\mathbb{S}^{n-1}$ by an arbitrary closed minimal hypersurface $\Sigma^{n-1} \subset \mathbb{S}^n$?}
\medskip

For the equator $\mathbb{S}^{n-1}$, computing the volume of its tubular neighborhood $T_\delta(\mathbb{S}^{n-1})$ is straightforward due to total geodesic symmetry. However, an arbitrary closed minimal hypersurface $\Sigma^{n-1} \subset \mathbb{S}^n$ may possess highly complex topological and geometric structures. Consequently, direct computation of $\Vol(T_\delta(\Sigma^{n-1}))$ is generally unfeasible. In  \cite{gray}, and \cite{gray2},  we have classical estimations for the volume of tubes around submanifolds under extrinsic and intrinsic curvature restrictions.

As an alternative way to the estimations developed in the above references, we introduce an indirect strategy based in the results obtained in \cite{BessaGimenoPalmer2026}. This strategy will uses two primary tools: the first is the {\em  Brownian Motion on the sphere, and its associated  Mean Exit Time (MET) function}. In this way, stochastic particles trapped inside a domain serve as indirect detectors of domain volume. If volume concentrates heavily within a region, the average time required for Brownian motion to escape that region must increase accordingly.
   
   The second tool is the {\em Schwarz Symmetrization}.  This is a measure-theoretic and analytic technique allowing comparison between PDE solutions on arbitrary domains in a Riemannian manifold $M$ and explicit radial solutions on model geodesic balls in $\mathbb{S}^n$.
   
   As a consequence of this combination among elliptic PDEs (via the Poisson problem for Mean Exit Time), spherical Schwarz symmetrization (transferring local PDE analytical bounds into global geometric inequalities), and classical measure concentration inequalities, we conclude that the Fat Equator Effect is not a special feature of totally geodesic submanifolds, but rather a  property of closed minimal hypersurfaces in the sphere, namely, the central result of this paper establishes that minimality alone is sufficient to recover measure concentration in full generality.

\begin{theorem}[Fat Minimality Effect] \label{thm:main_intro}\

Let $\mathcal{F}=\{\varphi_n: (\Sigma^{n-1}, \varphi_n^*(g_{\mathbb{S}^{n}})) \rightarrow (\mathbb{S}^{n},g_{\mathbb{S}^{n}})\}_{n\in \ene}$ be a family of  closed and minimal isometric immersions. Then, for every positive and sufficiently small  $\delta \in \erre^+$ and every $\epsilon \in (0, \pi/2 - \delta)$, we have that there exists $n_\delta$ such that for every $n \geq n_\delta$,
\begin{equation}
    \Vol(T_{\delta+\epsilon}(\Sigma^{n-1})) \ge \left( 1 - \sqrt{\frac{\pi}{8}} e^{-\epsilon^2 (n-1)/2} \right) \Vol(\mathbb{S}^n).
\end{equation}
\noindent 
Consequently, for every $r \in (0, \pi/2)$,
\begin{equation}\label{limit}
    \lim_{n \to \infty} \frac{\Vol(T_r(\Sigma^{n-1}))}{\Vol(\mathbb{S}^n)} = 1.
\end{equation}
\end{theorem}

\begin{remark}\label{goodef2}
Concentration of measure in tubes around submanifolds occurs as the dimension of these immersions increases toward infinity. Consequently, the existence of a positive radius that allows these tubes to be defined for all dimensions must be guaranteed, rendering the limit set forth in \eqref{limit} well-defined, (see remark \ref{rem1}). This prior consideration is what underlies the expression {\em sufficiently small  $\delta \in \erre^+$}, that we have used in the statement of the above theorem, avoiding some technicalities concerning its precise meaning, which is related with the family of immersions $\mathcal{F}$ in a way we shall explain in the next section \S.2, (see definition \ref{goodef3}). We shall state in a precise way Theorem \ref{thm:main_intro} in Theorem \ref{thm_main} in section \S.3.
\end{remark}

To prove this theorem, we shall see as the Mean Exit Time (MET)  function,  acts as an accurate, indirect volume detector: a large volume implies a large MET, whereas a small MET certifies a small volume.


\subsection{Outline}\

The structure of the paper can be outlined as follows: after the first section \S.1, (the {\em Introduction}), where we explain the context of the problem, in the second section, \S.2 {\em Preliminaires}, we describe the notion of tubular neighbourhood of a submanifold in a Riemannian manifold, and the two instruments we shall use to prove the main result: the mean exit time function of the Brownian motion and the Schwarz symmetrization on the sphere. The third section is devoted to the proof of our theorem.


\section{Preliminaries}\
\medskip

In this section, we review the definition of tubular neighbourhood around a submanifold, and the  two main technical tools required for the proof of Theorem \ref{thm:main_intro}: the Mean Exit Time function of Brownian motion and Schwarz symmetrization on Riemannian manifolds.

\subsection{Tubes around submanifolds}\label{tubes}\
\medskip

 We need to recall some facts about the geometry of tubes around submanifolds, following the book \cite{gray}. 
 
 Let $\varphi \colon \!(\Sigma, \varphi^*g) \to (M,g)$ be an isometric embedding of a complete Riemannian $m$-manifold $\Sigma$  into a Riemannian $n$-manifold $(M,g)$, (we shall refer $\Sigma$ as a {\em submanifold} of $M$). Let us identify $\Sigma$ with the zero section  $s\colon \Sigma \to T\Sigma^{\perp}$  of the normal bundle. The exponential map $\exp^{\perp}\colon T\Sigma^{\perp}\to M$ is a diffeomorphism when restricted to an open neighborhood $\Omega_{\Sigma}$ of $ T\Sigma^{\perp}$ containing $\Sigma$. To describe this set $\Omega_{\Sigma}$ observe that for every $q\in \exp^{\perp}(\Omega_{\Sigma})$ there exists a unique minimal geodesic  $\gamma_{\xi}\colon [0, \ell]\to M $ joining  $p=\gamma_{\xi}(0)\in \Sigma$ to $\gamma_{\xi}(\ell)=q$ with  $\gamma_{\xi}'(0)=\xi\in  \big(T_{\gamma_{\xi}(0)}\Sigma\big)^{\perp}$. The geodesic $\gamma_{\xi}$  realize the distance $\ell={\rm dist}_{_M}(\Sigma, q)$ between $q$ and $\Sigma$.

\begin{definition}\

Let us define the focal distance $c\colon ST\Sigma^{\perp}=\{(p,\xi)\in T\Sigma^{\perp}\colon \Vert \xi\Vert=1\}\to \mathbb{R}$  as $$c(p, \xi)=\sup_{t>0}\{t\colon {\rm dist}_{_M}(p, \gamma_{\xi} (t))=t\} $$
\noindent Then we have that $$ \exp^{\perp}\big(\Omega_{\Sigma}\big)= \{ \exp^{\perp}(p,s\xi)=\gamma_{\xi}(s),\,\,\, (p,\xi)\in ST\Sigma^{\perp}\,,\,0\leq s\leq c(p,\xi)\}.  $$ 
\end{definition}

\begin{definition}\label{minfoc}

The minimal focal distance $\mathbb{C}_M(\Sigma)$ of $\Sigma$ in $M$ is defined as $$\mathbb{C}_M(\Sigma)=\inf\{c(p, \xi), (p, \xi) \in ST\Sigma^{\perp}\}.$$ 
\end{definition}

\begin{definition}\

Let us consider $\varphi: (\Sigma^m,\varphi^*g) \longrightarrow (M^n, g)$ a submanifold of a Riemannian manifold $M$, and a number  $r\leq \mathbb{C}_M(\Sigma)$. The {\em tube of radius} $r$ about $\Sigma$ in $M$ is the set
\begin{equation*}
  T_r(\Sigma):=\left\{ x \in M\, :\, \text{dist}_M(x, \Sigma) \leq r \right\}.
\end{equation*}

\noindent and {\em the tubular hypersurface} of radius $r$ about $\Sigma$ in $M$ is the set
\begin{equation*}
  \partial T_r(\Sigma):=\left\{ x \in M\, :\, \text{dist}_M(x, \Sigma) = r \right\}.
\end{equation*}
\end{definition}
\begin{remark}\

 If $M$ is complete and $\Sigma$ is closed then $\mathbb{C}_M(\Sigma) >0$. Thus  we have  for all  $0\leq r\leq \mathbb{C}_M(\Sigma)$ that
\begin{equation*}
T_r(\Sigma)=\exp^{\perp}\Big(\left\{(p, \xi)\in T\Sigma^\bot, \text{ with }\, \Vert \xi\Vert \leq r\right\}\Big) \subseteq  \exp^{\perp}\left(\Omega_\Sigma\right).
\end{equation*}
\noindent so, as $\Sigma$ is closed, the points in the tube of radius $r$ are those whose distance to the submanifold is less than or equal to $r$, that is, those that can be reached by a minimizing geodesic that starts from the submanifold orthogonally. 
\end{remark}
\begin{definition}\label{goodef3}\

Let us consider $\mathcal{F}:=\{\varphi_n:(\Sigma^{n-1},\varphi_{n}^*g_{\mathbb{S}^n})\rightarrow (\mathbb{S}^n, \varphi^*g_{\mathbb{S}^n}) \}_{n\in \ene}$ a family of codimension $1$ isometric immersions in the sphere of sectional curvature $1$. The minimal focal distance of $\Sigma^{n-1}$ in $\mathbb{S}^{n}$, $\mathbb{C}_{\mathbb{S}^{n}}(\Sigma^{n-1})$, (see definition \ref{minfoc}) depends on the immersion $\varphi_n$, (on the metric $\varphi_{n}^*g_{\mathbb{S}^n}$), and, in particular, of the dimension $n$. We shall denote it as $\mathbb{C}_{\mathbb{S}^{n}}(\Sigma^{n-1})=\mathbb{C}_n$. Then, we shall define the {\em minimal focal distance associated to the family of immersions $\mathcal{F}$} as the number
$$\mathbb{C}(\mathcal{F}):=\inf_{n \in \ene} \{\mathbb{C}_n\} \geq 0$$ 


As we have mentioned before in remark \ref{rem1}, when we consider the family of totally geodesic immersions $\tilde\varphi_n: (\mathbb{S}^{n-1}, \tilde\varphi_{n}^*(g_{\mathbb{S}^{n}}))\rightarrow (\mathbb{S}^{n},g_{\mathbb{S}^{n}})$, the minimal focal distance in $\mathbb{S}^{n}$, 
$$\tilde{\mathbb{C}}_n=\mathbb{C}_{\mathbb{S}^{n}}(\mathbb{S}^{n-1})$$
\noindent  do not depends on the dimension $n$ and we have, for all $n \in \ene$, that
$$\tilde{\mathbb{C}}_n=\mathbb{C}_{\mathbb{S}^{n}}(\mathbb{S}^{n-1})=\inf\{c(p, \xi), (p, \xi) \in ST\mathbb{S}^{{n-1}^{\perp}}\}=\frac{\pi}{2}$$ 
\end{definition}
\begin{remark}
To simplify the notation and provided there is no risk of confusion, throughout the rest of the paper we will use the notation $\Sigma \subseteq M$ or $\varphi: \Sigma \to M$ to refer to isometrically immersed submanifolds.
\end{remark}
\subsection{The Mean Exit Time (MET) Function}\
\medskip

Let $(M, g)$ be an $n$-dimensional Riemannian manifold, and let $\Omega \subset M$ be a open connected domain with smooth boundary $\partial\Omega$. Let $X_t$ denote Brownian motion on $M$ initialized at $x \in \Omega$, and let $\tau_\Omega^x = \inf\{ t > 0 : X_t \notin \Omega \}$ be the first exit time of $X_t$ from $\Omega$.

The \emph{Mean Exit Time (MET)} function $E_\Omega(x) := \E^x[\tau_\Omega^x]$ represents the expected time spent by Brownian motion inside $\Omega$ before hitting the boundary $\partial\Omega$. It is well known that $E_\Omega \in C^\infty(\Omega) \cap C^0(\overline{\Omega})$ is the unique classical solution to the Poisson boundary value problem:
\begin{equation} \label{eq:poisson_met}
    \begin{cases}
        \Delta^M E_\Omega = -1 & \text{in } \Omega, \\
        E_\Omega = 0 & \text{on } \partial\Omega,
    \end{cases}
\end{equation}
where $\Delta^M$ denotes the Laplace--Beltrami operator on $M$.

Intuitively, $E_\Omega(x)$ measures how long a stochastic particle remains trapped inside $\Omega$. If $\Omega = T_\delta(\Sigma^{n-1})$ concentrates almost the entire volume of $M$, the exit time $E_\Omega$ should be large, as it has been proved in \cite{BessaGimenoPalmer2026}.

We have estimated the MET on two canonical model domains in the standard sphere $\mathbb{S}^n$:
\begin{enumerate}
    \item The equatorial tube $T_\delta(\mathbb{S}^{n-1}) \subset \mathbb{S}^n$ of radius $\delta \in (0, \pi/2)$.
    \item The polar geodesic ball $B_{\pi/2 - \delta}(p_N) \subset \mathbb{S}^n$ centered at the north pole $p_N$.
\end{enumerate}

As established in Theorem 1.1 in \cite{BessaGimenoPalmer2026}, these model domains display drastically contrasting asymptotic behaviors as $n \to \infty$:

\begin{theorem}[MET behavior in the sphere, \cite{BessaGimenoPalmer2026}]\label{thmmain}
Let us consider the tube  $T_\delta(\mathbb{S}^{n-1}) \subseteq \mathbb{S}^{n}$, of radius $\delta \in (0,\pi/2)$, around the totally geodesic submanifold $\varphi: \mathbb{S}^{n-1} \rightarrow \mathbb{S}^{n}$, and the geodesic ball $B_{\pi/2-\delta}(p_N) \subseteq \mathbb{S}^{n}$ centered at the north pole $p_N$. Then, we have the following statements:
\begin{itemize}
\item[i)] The mean exit time function $E_{T_\delta(\mathbb{S}^{n-1})} $ defined on the tube $T_{\delta}(\mathbb{S}^{n-1})$, is bounded from below by the function $F_{\delta,n}\circ s$ as
$$
\begin{aligned}
E_{T_\delta(\mathbb{S}^{n-1})} (x)=E_{T_\delta(\mathbb{S}^{n-1})} (s(x))&\geq (F_{\delta,n}\circ s)(x)\,\,\,\forall x \in T_{\delta}(\mathbb{S}^{n-1})
\end{aligned} $$
\noindent where $s(x)=\mathrm{dist}_{\mathbb{S}^{n}}(x,\mathbb{S}^{n-1})$ and  $F_{\delta,n}: (0,\delta) \rightarrow \mathbb{R}$ is the function defined as
$$F_{\delta,n}(r)=\frac{(\delta-r)}{\sin(\delta)}\cdot\frac{1-\cos^{n}(r)}{n\cos^{n-1}(r)}$$
\noindent which satisfies  that
$$ \lim_{n\to \infty}F_{\delta,n}(r)=\infty \,\,\,\forall r \in (0,\delta)$$
\noindent Hence
$$\max E_{T_\delta(\mathbb{S}^{n-1})} \longrightarrow \infty \quad \text{as } n \to \infty$$

\item[ii)] The mean exit time function $E_{B_{\pi/2 - \delta}(\mathbb{S}^n)}$ from the ball  $B_{\pi/2-\delta}(p_N) \subseteq \mathbb{S}^{n}$ is uniformly  bounded from above as $$E_{B_{\pi/2 - \delta}(\mathbb{S}^n)}(y)=E_{B_{\pi/2 - \delta}(\mathbb{S}^n)}(t(y))\leq G_{\delta,n}:= \frac{1-\cos(\pi/2-\delta)}{n\cos(\pi/2-\delta)}\cdot$$\label{thmprelim}
\noindent where $t(y)=\mathrm{dist}_{\mathbb{S}^{n}}(y,p_N)$, and we have that the uniform bound $G_{\delta,n}$ satisfies
 $$\lim_{n\to \infty}G_{\delta,n}=0\, .$$
 \noindent Hence 
 $$\max E_{B_{\pi/2 - \delta}(\mathbb{S}^n)} \longrightarrow 0 \quad \text{as } n \to \infty$$
 \end{itemize}
\end{theorem} 


When we consider the mean exit time function defined on a tubular neighbourhood around a minimal hypersurface isometrically immersed in a closed Riemannian manifold with Ricci curvature bounded from below by the corresponding Ricci curvature of the unit sphere, we obtain the following comparison theorem, (Theorem 1.2 in \cite{BessaGimenoPalmer2026}):

\begin{theorem}[MET Comparison Theorem \cite{BessaGimenoPalmer2026}] \label{thm:met_comp}
Let $M^n$ be a Riemannian manifold with $\text{Ric}_M \ge n-1$, and let $\Sigma^{n-1} \subset M$ be a closed minimal hypersurface. Then, given $\delta \leq \mathbb{C}_n$, the MET function defined on the tubular neighborhood $T_\delta(\Sigma^{n-1})$ is bounded below pointwise by the MET solution transplanted from the model equatorial tube in $\mathbb{S}^n$:
\begin{equation}
    E_{T_\delta(\Sigma^{n-1})}(x) \ge E_{T_\delta(\mathbb{S}^{n-1})}(r_\Sigma(x)) \quad \forall x \in T_\delta(\Sigma^{n-1}),
\end{equation}
where $r_\Sigma(x) := \text{dist}_M(x, \Sigma^{n-1})$.
\end{theorem}

Applied directly to the mean exit time function defined on a tubular neighbourhood around a minimal hypersurface isometrically immersed in the sphere $\mathbb{S}^n$, Theorem \ref{thm:met_comp} yields the following immediate consequence:

\begin{corollary}[MET Divergence on Minimal Spherical Tubes] \label{cor:met_divergence}\

Let $\Sigma^{n-1} \subset \mathbb{S}^n$ be a closed minimal hypersurface. Then, if $\delta \leq \mathbb{C}_n=\mathbb{C}_{\mathbb{S}^{n}}(\Sigma^{n-1})$, we have:
\begin{equation}\label{ineq1}
    E_{T_\delta(\Sigma^{n-1})} \ge E_{T_\delta(\mathbb{S}^{n-1})}
\end{equation}
\noindent pointwise. 

Consequently, if we consider a family of closed and minimal isometric immersions $\mathcal{F}=\{\varphi_n:\Sigma^{n-1} \rightarrow \mathbb{S}^n\}_{n\in \ene}$, and we assume that the number $\mathbb{C}(\mathcal{F}):=inf_{n \in \ene} \{\mathbb{C}_n\} > 0$, then, for all $\delta \leq \mathbb{C}(\mathcal{F})$, the maximum of the mean exit time function in $T_\delta(\Sigma^{n-1})$ grows without bound as the dimension increases:
\begin{equation}\label{ineq2}
    \max E_{T_\delta(\Sigma^{n-1})} \longrightarrow \infty \quad \text{as } n \to \infty.
\end{equation}
\end{corollary}

\subsection{Schwarz Symmetrization}\
\medskip

We present in this subsection a notion of \emph{ relative Schwarz-symmetrization} adapted to the case given by a closed manifold.

\begin{definition}\label{def:Symmcompact}\

	Let $(M^n, g)$ be a closed Riemannian manifold. Suppose that $D \subseteq M$ is a precompact open connected domain in $M^n$. Then the \emph{relative $\mathbb{S}^n$- symmetrization of \, $D$} is denoted by $D^{*}$ and is defined to be the unique $s(D)$-ball in $\mathbb{S}^n$,  $D^{*}:=B^{\mathbb{S}^n}_{s(D)}(p_N)$, (where we fix $p_N$ the north-pole of $\mathbb{S}^n$), having the same relative volume:
	
	\begin{equation*}
		\Vrel^{M}(D)=\frac{\Vol(D)}{\Vol(M)}=\frac{\Vol\left(B_{s(D)}\right)}{\Vol(\mathbb{S}^n)}=\Vrel^{\mathbb{S}^n}(B^{\mathbb{S}^n}_{s(D)})
	\end{equation*}
	\end{definition}
Once we have defined the symmetrization of a domain, we are going to define the Schwarz symmetrization of a function.  To do that, we start defining the relative volume of the super-level sets $D(t) := \{ x \in D : u(x) \ge t \}$ as
    \begin{equation}
        \mu_M(t) := \Vrel^{M}(D(t)) = \frac{\Vol(\{ x \in D : u(x) \ge t \})}{\Vol(M)}.
    \end{equation}

Once this quantity $\mu_M(t)$, (called the {\em distribution function}), is defined for each $t \in (0,  \operatorname{ess\,sup}_D u)$, we define the {\em decreasing arrangement} of $u$ as

\begin{definition}[Decreasing rearrangement]\label{decrea}\

Define $u^\sharp : [0, \Vrel(D)] \to [0, \infty)$ by
    \begin{equation}
        u^\sharp(s) := 
        \begin{cases}
            \operatorname{ess\,sup}_D u & \text{if } s = 0, \\
            \inf \{ t \ge 0 : \mu_M(t) \le s \} & \text{if } s > 0.
        \end{cases}
    \end{equation}
    
     Note that, by its definition, $u^\sharp$ is non-increasing and attains its maximum at $s=0$, 
$$max_{[0, \Vrel(D)]} u^\sharp=u^\sharp(0)=\operatorname{ess\,sup}_D u $$
\end{definition}

Then, the Schwarz symmetrization of the function $u: D \rightarrow \erre^+$ is defined in the following way:
\begin{definition}[Schwarz-Symmetrization of a Function]
Let $u : D \to [0, \infty)$ be a nonnegative measurable function on $D \subseteq M$. Its $\mathbb{S}^n$-Schwarz symmetrization $u^* : D^* \to [0, \infty)$ is defined as the unique radial non-increasing function with respect to $r_{p_N}(x^*) = \text{dist}_{\mathbb{S}^n}(x^*, p_N)$ given by
    \begin{equation}
        u^*(x^*) = u^\sharp \left( \frac{\Vol(B_{r_{p_N}(x^*)}^{\mathbb{S}^n}(p_N))}{\Vol(\mathbb{S}^n)} \right)=u^\sharp(\Vrel^{\mathbb{S}^n}(B_{r_{p_N}(x^*)}^{\mathbb{S}^n}(p_N)).
    \end{equation}
\end{definition}
\begin{remark}
Equivalently, $u^*$ is the unique radial, non-increasing function on $D^*$ satisfying
\begin{equation}
    \Vrel^{\mathbb{S}^n}(\{ u^* \ge t \}) = \Vrel^M(D(t)) \quad \forall t \ge 0.
\end{equation}
\end{remark}

\begin{proposition}[Properties of Schwarz Symmetrization]\

The Schwarz symmetrization of a function satisfies the following properties:
\begin{enumerate}
    \item $u^*$ is radial and non-increasing with respect to the distance from the target pole $p_N$.
    \item Symmetrization preserves relative $L^1$ integrals:
    \begin{equation}
        \frac{1}{\Vol(M)} \int_D u \, d\sigma = \frac{1}{\Vol(\mathbb{S}^n)} \int_{D^*} u^* \, d\tilde{\sigma}.
    \end{equation}
    \item For any two smooth nonnegative functions $u_1, u_2$ on $D$,
    \begin{equation}
        u_1 \le u_2 \implies u_1^* \le u_2^*.
    \end{equation}
\end{enumerate}
\end{proposition}

There is another property which deserves to be highlited because will play a key r\^{o}le in the proof of the main theorem.

\begin{proposition}\label{maximums}\

Let $(M,g)$ be a compact Riemannian manifold of dimension $n$ without boundary ($\partial M = \emptyset$), and let $D \subset M$ be a bounded open domain with smooth boundary $\partial D$. Let $E_D \in C^2(\overline{D})$ be the unique classical solution to the Poisson problem
$$
\begin{cases}
-\Delta_g E_D = 1 & \text{in } D, \\
E_D = 0 & \text{on } \partial D.
\end{cases}
$$
Let $E_D^*$ be the Schwarz symmetrization of $E_D$ on the unit sphere $\mathbb{S}^n(1)$, constructed via the relative volume measure. Then,
$$
\max_{x \in \overline{D}} E_D(x) = \max_{x \in \overline{D}^*} E_D^*(x).
$$
\end{proposition}

\begin{proof}


To start the proof, and as  the data function $f \equiv 1$ is smooth ($C^\infty(D)$) and $\partial D$ is of class $C^\infty$, standard boundary elliptic regularity ensures that $E_D \in C^\infty(D) \cap C^2(\overline{D})$. 

Since $\overline{D}$ is compact and $E_D$ is continuous, $E_D$ attains its maximum on $\overline{D}$. Furthermore, since $-\Delta_g E_D = 1 > 0$ in $D$ and $\left.E_D\right|_{\partial D} = 0$, the Strong Maximum Principle implies that $E_D(x) > 0$ for all $x \in D$ and that $E_D$ achieves its maximum at an interior point $x_0 \in D$. Define
$$
M_0 := \max_{x \in \overline{D}} E_D(x) = E_D(x_0) > 0.
$$


By definition of the decreasing rearrangement, (see Definition \ref{decrea}), $E_D^\sharp(0)$ represents the essential supremum of $E_D$ on $D$:
$$
E_D^\sharp(0) = \mathrm{ess\,sup}_{x \in D} E_D(x).
$$

To show that $E_D^\sharp(0) = M_0$, fix any $\varepsilon > 0$ and consider the superlevel set:
$$
\Omega_{M_0 - \varepsilon} = \{ x \in D : E_D(x) > M_0 - \varepsilon \}.
$$
Since $E_D$ is continuous at $x_0$ and $E_D(x_0) = M_0$, there exists an open neighborhood $U \subset D$ of $x_0$ such that $U \subseteq \Omega_{M_0 - \varepsilon}$. On a Riemannian manifold, every non-empty open set has strictly positive volume measure, which implies:
$$
\mu_M\left( \Omega_{M_0 - \varepsilon} \right) \ge \mu_M(U) > 0.
$$
Consequently, for any $0<M_0-\epsilon=t < M_0$, the measure $\mu_M(\{ x \in D : E_D(x) > t \})$ is strictly positive. Therefore, by definition of $M_0$, we have that
$\inf \{ t \ge 0 : \mu_M(\{ E_D > t \}) = 0 \} = M_0$, establishing that:
\begin{equation}\label{eqt1}
E_D^\sharp(0) = M_0 = \max_{x \in \overline{D}} E_D(x).
\end{equation}


Now, let us fix the North pole $p_N \in \mathbb{S}^n$. Let $v : [0, \pi] \to [0, 1]$ be the relative volume function of geodesic balls in $\mathbb{S}^n$:
$$
v(r) = \frac{\mathrm{Vol}_{\mathbb{S}^n}(B_r(p_0))}{\mathrm{Vol}(\mathbb{S}^n)} = \frac{\int_0^r \sin^{n-1}(\theta) \, d\theta}{\int_0^\pi \sin^{n-1}(\theta) \, d\theta}.
$$
The function $v(r)$ is continuous, strictly increasing, and satisfies $v(0) = 0$. The symmetrized domain $D^* \subset \mathbb{S}^n$ is the geodesic ball $B_{R^*}(p_N)$ with radius $R^* = v^{-1}\left(\mu_M(D)\right)$. 

The Schwarz symmetrization $E_D^* : \mathbb{S}^n \to [0, \infty)$ is defined by:
$$
E_D^*(x) = E_D^\sharp \left( v(r_{p_{N}}(x)) \right).
$$

Because $E_D^\sharp$ is non-increasing on $[0, \mu_M(D)]$ and $v(r)$ is strictly increasing in $r = r_{p_{N}}$, the composite function $r \mapsto E_D^\sharp(v(r))$ achieves its maximum when $r$ is minimized, which occurs at $r = 0$ (the pole $x = p_N$):
\begin{equation}\label{eqt2}
\max_{x \in \overline{D}^*} E_D^*(x) = \max_{[0,\Vrel^m(D)]}(E^\sharp \circ v)=E_D^*(p_N) = (E_D^\sharp\circ v)(0)=E_D^\sharp(v(0)) = E_D^\sharp(0).
\end{equation}


Combining the equalities established in \eqref{eqt1} and \eqref{eqt2}, we obtain
$$
\max_{x \in \overline{D}^*} E_D^*(x) = E_D^\sharp(0) = \max_{x \in \overline{D}} E_D(x).
$$
\end{proof}

A key result for comparing Poisson solutions under symmetrization is the Talenti--Szeg\H{o} Comparison Theorem on closed Riemannian manifolds stablished by D. Colladay, L. Langford and P. McDonald in \cite{ColladayLangfordMcDonald2018}:

\begin{theorem}[Talenti--Szeg\H{o} Comparison Theorem \cite{ColladayLangfordMcDonald2018}] \label{thm:talenti}\

Let $(M^n, g)$ be a closed Riemannian manifold satisfying $\text{Ric}_M \ge n-1$. Consider the Poisson problem on $D \subseteq M$ with $f \in L^2(D)$:
\begin{equation}
    \begin{cases}
        \Delta^M u + f = 0 & \text{in } D, \\
        u|_{\partial D} = 0,
    \end{cases}
\end{equation}
and its corresponding symmetrized problem on the geodesic ball $D^* \subseteq \mathbb{S}^n$:
\begin{equation}
    \begin{cases}
        \Delta^{\mathbb{S}^n} v + f^* = 0 & \text{in } D^*, \\
        v|_{\partial D^*} = 0.
    \end{cases}
\end{equation}
Then, the rearranged solution $u^*$ satisfies pointwise:
\begin{equation}
    u^* \le v \quad \text{on } D^*.
\end{equation}
\end{theorem}

Now, we are going to apply Theorem \ref{thm:talenti} to the Mean Exit Time problem defined on a tube $T_\delta(\Sigma^{n-1})$, around a closed minimal hypersurface $\Sigma^{n-1} \subseteq \mathbb{S}^n$. In this case, the data function is $f=1$, so the symmetrized problem in the sphere is posed in a geodesic ball $D^* = B_{r(\delta)}^{\mathbb{S}^n} \subseteq \mathbb{S}^n$ with the same relative volume than $T_\delta(\Sigma^{n-1})$, and, as $f^*=1$, its solution is the mean exit time function on this geodesic ball $D^* = B_{r(\delta)}^{\mathbb{S}^n}$. Hence, we have the following upper bound:

\begin{corollary}[Upper Bound on Minimal Tubes] \label{cor:upper_bound_met}\

Let $\Sigma^{n-1} \subseteq \mathbb{S}^n$ be a closed minimal hypersurface, and let $E_{T_\delta(\Sigma^{n-1})}$ be the mean exit time function of the tube $T_\delta(\Sigma^{n-1})$, with $\delta \leq \mathbb{C}_n$. Let $D^* = B_{r_n(\delta)}^{\mathbb{S}^n} \subseteq \mathbb{S}^n$ be the geodesic ball with the same (relative) volume as $T_\delta(\Sigma^{n-1})$. Then:
\begin{equation}
    E_{T_\delta(\Sigma^{n-1})}^* \le E_{B_{r_{n}(\delta)}^{\mathbb{S}^n}}.
\end{equation}

Now, let us consider a family of closed and minimal isometric immersions $\mathcal{F}=\{\varphi_n:\Sigma^{n-1} \rightarrow \mathbb{S}^n\}_{n\in \ene}$, and let us assume that the number $\mathbb{C}(\mathcal{F}):=inf_{n \in \ene} \{\mathbb{C}_n\} > 0$.  Then, if we assume, given $\delta \leq \mathbb{C}(\mathcal{F})$, that, for all $n \in \ene$, we have
 $$\Vol(T_\delta(\Sigma^{n-1})) \le \frac{1}{2} \Vol(\mathbb{S}^n)$$
 \noindent we conclude that the symmetrized ball $B_{r_{n}(\delta)}^{\mathbb{S}^n}$ is contained in a polar hemisphere, and
\begin{equation}\label{tozero}
    E_{T_\delta(\Sigma^{n-1})}^* \le E_{B_{r(\delta)}^{\mathbb{S}^n}} \longrightarrow 0 \quad \text{as } n \to \infty.
\end{equation}
\end{corollary}
\begin{proof}
Let us fix $\delta \leq \mathbb{C}(\mathcal{F})$.  As we assume that $\Vol(T_\delta(\Sigma^{n-1})) \le \frac{1}{2} \Vol(\mathbb{S}^n)$ for all $n \in \ene$, and $B_{r_n(\delta)}^{\mathbb{S}^n} \subseteq \mathbb{S}^n$ is the geodesic ball with the same (relative) volume as $T_\delta(\Sigma^{n-1})$ then $r_{n}(\delta) \le \pi/2\,\,\forall n \in \ene$. 

Now, to obtain inequality $E_{T_\delta(\Sigma^{n-1})}^* \le E_{B_{r_{n}(\delta)}^{\mathbb{S}^n}}$, we apply Theorem \ref{thm:talenti}, and to get conclusion \eqref{tozero}, we apply assertion ii) of Theorem \ref{thmmain}.
\end{proof}

\section{Measure Concentration Around Minimal Hypersurfaces}\
\medskip

We now present the restatement and complete proof of our main result.

\begin{theorem}[Main Theorem]\label{thm_main}\

Let $\mathcal{F}=\{\varphi_n:\Sigma^{n-1} \rightarrow \mathbb{S}^n\}_{n\in \ene}$ be a family of closed and minimal isometric immersions, and let us assume that the minimal focal distance associated to the family is
$$\mathbb{C}(\mathcal{F}):=\inf_{n \in \ene} \{\mathbb{C}_n\} > 0$$

Then, given $\delta \leq \mathbb{C}(\mathcal{F})$, and given $\epsilon \in (0, \pi/2 - \delta)$, there exists $n_\delta \in \ene$ such that, if $n \geq _\delta$, we have
\begin{equation} \label{eq:main_bound}
    \Vol(T_{\delta+\epsilon}(\Sigma^{n-1})) \ge \left( 1 - \sqrt{\frac{\pi}{8}} e^{-\epsilon^2 (n-1)/2} \right) \Vol(\mathbb{S}^n).
\end{equation}
Hence, for this familiy of closed and minimal immersions, we have that
\begin{equation}
    \lim_{n \to \infty} \frac{\Vol(T_r(\Sigma^{n-1}))}{\Vol(\mathbb{S}^n)} = 1.
\end{equation}
\end{theorem}

\begin{proof}\

Let us consider the family of closed and minimal isometric immersions $\mathcal{F}=\{\varphi_n:\Sigma^{n-1} \rightarrow \mathbb{S}^n\}_{n\in \ene}$, and let us assume that the minimal focal distance associated to the family $\mathbb{C}(\mathcal{F}):=inf_{n \in \ene} \{\mathbb{C}_n\} > 0$.  Let us consider now $\delta \leq \mathbb{C}(\mathcal{F})$.

The proof proceeds in three main steps.

\medskip

First, and by inequality  \eqref{ineq1} in Corollary \ref{cor:met_divergence}, since $\Sigma^{n-1} \subset \mathbb{S}^n$ is closed and minimal, the MET solution $E_{T_\delta(\Sigma^{n-1})}$ satisfies, for all $n \in \ene$,
\begin{equation}
    E_{T_\delta(\Sigma^{n-1})} \ge E_{T_\delta(\mathbb{S}^{n-1})}.
\end{equation}
Applying the order-preserving property of Schwarz symmetrization yields
\begin{equation}
    E_{T_\delta(\Sigma^{n-1})}^* \ge E_{T_\delta(\mathbb{S}^{n-1})}^* \quad \forall n \in \mathbb{N}.
\end{equation}
Taking the maximum on both sides, and applying Proposition \ref{maximums}, (because, as $\delta$ is less than the focal distance of $\Sigma$ in $\mathbb{S}^n(1)$, the tube $T_\delta(\Sigma^{n-1})$ is an open and bounded domain in $\mathbb{S}^n$ with smooth boundary $\delta T_\delta(\Sigma^{n-1})$), we obtain
\begin{equation} \label{eq:step1_div}
    \max E_{T_\delta(\Sigma^{n-1})}^* \ge \max E_{T_\delta(\mathbb{S}^{n-1})}^* = \max E_{T_\delta(\mathbb{S}^{n-1})} \longrightarrow \infty \quad \text{as } n \to \infty.
\end{equation}
Thus, the symmetrized MET sequence $\{ \max E_{T_\delta(\Sigma^{n-1})}^*\}_{n \in \ene}$ is a non-bounded sequence diverging to infinity, so it cannot contain infinite and convergent subsequences.


The second step of the proof consist in to get an estimation of the volume of the tube for large dimensions.  Given our family of closed and minimal isometric immersions $\mathcal{F}=\{\varphi_n:\Sigma^{n-1} \rightarrow \mathbb{S}^n\}_{n\in \ene}$, and given our radius $\delta \leq \mathbb{C}(\mathcal{F})$,  we claim that there exists $n_\delta \in \mathbb{N}$ such that for all $n \ge n_\delta$,
\begin{equation} \label{eq:step2_vol}
    \Vol(T_\delta(\Sigma^{n-1})) \ge \frac{1}{2} \Vol(\mathbb{S}^n).
\end{equation}
We establish \eqref{eq:step2_vol} by contradiction. Let us suppose that for all $n \in \ene$, there exists $m_n \geq n$ such that 
\begin{equation}\label{eq:step2_vol2}
\Vol(T_\delta(\Sigma^{{m_{n}}-1})) < \frac{1}{2} \Vol(\mathbb{S}^{m_{n}})
\end{equation} 
\noindent for all $n \in \mathbb{N}$.

Namely, inequality \eqref{eq:step2_vol2} holds along an unbounded subsequence of natural numbers $\{m_n\}_{n\in \ene}$. We are going to study the convergence of the subsequence $\{ \max E_{T_\delta(\Sigma^{m_{n}-1})}^*\}_{n \in \ene}$. 

Let $B_{r_{m_{n}}(\delta)}^{\mathbb{S}^{m_{n}}}$ be the symmetrization ball corresponding to $T_\delta(\Sigma^{m_{n}-1})$. Since $\Vol(T_\delta(\Sigma^{m_{n}-1})) < \frac{1}{2} \Vol(\mathbb{S}^{m_{n}})$, the radius $r_{m_{n}}(\delta)$ satisfies $r_{m_{n}}(\delta) < \pi/2$, meaning that $B_{r_{m_{n}}(\delta)}^{\mathbb{S}^{m_{n}}}$ is strictly contained in a polar hemisphere. By Corollary \ref{cor:upper_bound_met},
\begin{equation}
    E_{T_\delta(\Sigma^{m_{n}-1})}^* \le E_{B_{r_{m_{n}}(\delta)}^{\mathbb{S}^{m_{n}}}} \longrightarrow 0 \quad \text{as } n \to \infty.
\end{equation}
This implies that $\max E_{T_\delta(\Sigma^{m_{n}-1})}^* \to 0$ as $n \to \infty$, namely, that we have a convergent-to-zero subsequence $\{ \max E_{T_\delta(\Sigma^{m_{n}-1})}^*\}_{n \in \ene}$ which directly contradicts the divergence result \eqref{eq:step1_div} established in the first step of the proof. 

Therefore, there must exist $n_\delta \in \mathbb{N}$ such that $\Vol(T_\delta(\Sigma^{n-1})) \ge \frac{1}{2} \Vol(\mathbb{S}^n)$ for all $n \ge n_\delta$.


Finally, we apply the Measure Concentration Theorem. Set $A := T_\delta(\Sigma^{n-1}) \subset \mathbb{S}^n$. For $n \ge n_\delta$, the second step of the proof \eqref{eq:step2_vol} guarantees that $\Vol(A) \ge \frac{1}{2} \Vol(\mathbb{S}^n)$.

The $\epsilon$-fattening of $A$ is precisely the larger tubular neighborhood:
\begin{equation}
    A_\epsilon = \left\{ x \in \mathbb{S}^n : \text{dist}_{\mathbb{S}^n}(x, T_\delta(\Sigma^{n-1})) < \epsilon \right\} = T_{\delta+\epsilon}(\Sigma^{n-1}).
\end{equation}
Applying the classical Measure Concentration Theorem on the sphere (Theorem 1.1) to $A$, we obtain, for all $\epsilon \in (0, \pi/2 - \delta)$:
\begin{equation}
    \frac{\Vol(T_{\delta+\epsilon}(\Sigma^{n-1}))}{\Vol(\mathbb{S}^n)} \ge 1 - \sqrt{\frac{\pi}{8}} e^{-\epsilon^2 (n-1)/2}.
\end{equation}
Taking $n \to \infty$ for any radius $r = \delta + \epsilon \in (0, \pi/2)$, we conclude that
\begin{equation}
    \lim_{n \to \infty} \frac{\Vol(T_r(\Sigma^{n-1}))}{\Vol(\mathbb{S}^n)} = 1,
\end{equation}
which completes the proof.
\end{proof}


\begin{thebibliography}{99}

\bibitem{Bengtsson}
I. Bengtsson and K. {\.Z}yczkowski.
\newblock {\em Geometry of quantum states. An introduction to quantum entanglement}.
\newblock Cambridge University Press, Cambridge, 2006.



\bibitem{BessaGimenoPalmer2026}
G.~P. Bessa, V.~Gimeno, and V.~Palmer,
\emph{Mean exit time bounds and volume concentration on Riemannian manifolds},
Accepted in Trans. Amer. Math. Soc., 2026.

\bibitem{ColladayLangfordMcDonald2018}
D.~Colladay, L.~Langford, and P.~McDonald,
\emph{Comparison theorems for the Poisson equation on Riemannian manifolds},
J. Geom. Anal., 28 (2018), 3121--3142.

\bibitem{GP}
V.~Gimeno and V.~Palmer,
\emph{Fat equator effect and Minimality in immersions and submersions of the Sphere}, Trans. London. Math. Soc., DOI:10.1112/tlm3.70008, 2025, 1-23.


\bibitem{gray}A. Gray, {\em Tubes,} second ed. Progress in Mathematics, Birkh\"{a}user Verlag, Basel, \textbf{221}, (2004). With a preface by Vicente Miquel.

\bibitem{gray2}
A. Gray
\emph{Comparison theorems for the volumes of tubes as generalizations of the Weyl tube formula },Topology, 21, (1982), 201-228.


\bibitem{MilmanSchechtman2001}
V.~D. Milman and G.~Schechtman,
\emph{Asymptotic Theory of Finite Dimensional Normed Spaces},
Lecture Notes in Mathematics, vol. 1200, Springer-Verlag, Berlin, 2001.

\end{thebibliography}
\end{document}